\documentclass{article}

\usepackage{setspace}
\usepackage{graphicx}
\usepackage{amsthm}
\usepackage{lipsum}
\usepackage{amssymb}
\usepackage{amsmath}
\usepackage{proof}
\usepackage{wasysym}
\usepackage{float}

\usepackage{enumitem}
\usepackage{amsthm}
\usepackage{amssymb}
\usepackage{latexsym}
\usepackage{microtype}
\usepackage{proof}
\usepackage[dvipsnames]{xcolor}
\theoremstyle{plain}
\newtheorem{theorem}{Theorem}[section]

\theoremstyle{definition}
\newtheorem{definition}{Definition}[section]

\theoremstyle{definition}

\newcommand{\impl}{\supset}
\newcommand{\et}{\wedge}
\newcommand{\vel}{\vee}
\newcommand{\fal}{\bot}

\newcommand{\non}{\mathord{\sim}}
\newcommand{\ergo}{\Rightarrow}

\newcommand{\calcg}{\mathbf{G}}

\newcommand{\amal}{\mathrm{Am}}

\newcommand{\lang}{\mathcal{L}}

\newcommand{\langg}{\lang_\gr}

\newcommand{\condr}{\triangleright}

\newcommand{\gr}{<}
\newcommand{\set}{\mathrm{\{\,\}}}

\title{A Direct Characterisation of Logical Grounds and a Decidability Proof}
\author{Francesco A. Genco\\LLC, University of Turin, Turin, Italy}
\begin{document}
\maketitle

\section{Introduction}

We present a standard calculus for logical grounding based on well-established grounding principles \cite{Schnieder2011,fin12guide,cor14,Correia2024} and provide a very direct characterisation of the provable grounding claims exclusively relying on the syntactic tree of the grounded formula. The technical features of the characterisation imply that the grounding relation induced by the calculus is decidable.

\section{The grounding calculus}

We define the grounding language that we will use.
\begin{definition}
The language $\lang$ is defined as follows:
\begin{itemize}
\item if $\alpha$ is one of the sentence letters $p,q,r,\dots$ or $\bot$, then $\alpha\in\lang$;
\item if $\alpha,\beta\in\lang$, then $\alpha\et \beta\in\lang$ and $\alpha\impl \beta\in\lang$;
\item if $\alpha_1, \dots , \alpha_n,\beta\in\lang$, then $\{\alpha_1, \dots , \alpha_n\}<\beta\in\langg$.
\end{itemize}
\end{definition}

We use Greek lowercase letters $\alpha, \beta,\gamma,\dots$ as metavariables for formulae and, for any formula $\alpha$, 
we use the abbreviation $\non\alpha$ to denote the formula $\varphi \impl\fal$.


\begin{figure}[ht]
\hrule
\medskip

\[
\infer{  \ergo \{\alpha,\beta\}<\alpha\et\beta}{}
\qquad 
\infer{  \ergo\{\alpha\}<\alpha\vel\beta}{}
\qquad 
\infer{  \ergo\{\beta\}<\alpha\vel\beta}{}
\qquad 
\infer[\ast]{  \ergo\{\alpha,\beta\}<\alpha\vel\beta}{}
\]\[ 
\infer{  \ergo\{\non\alpha\}<\non(\alpha\et\beta)}{}
\qquad 
\infer{  \ergo\{\non\beta\}<\non(\alpha\et\beta)}{}
\qquad
\infer[\ast]{  \ergo\{\non\alpha,\non\beta\}<\non(\alpha\et\beta)}{}
\]\[ 
\infer{  \ergo\{\non\alpha,\non\beta\}<\non(\alpha\vel\beta)}{}
\qquad 
\infer{  \ergo\{\alpha\}<\non\non\alpha}{}
\]

\medskip
\hrule
\caption{$0$-premiss rules.}
\label{fig:0-rules-gr}
\end{figure}

\begin{figure}[ht]
\hrule
\medskip


\[
\infer{    \ergo\{\alpha_1,\dots,\alpha_n,\beta_1,\dots,\beta_m\}<\alpha\et\beta}{  \ergo\{\alpha_1,\dots,\alpha_n\}<\alpha&\ergo\{\beta_1,\dots,\beta_m\}<\beta}
\]
\[ 
\infer{  \ergo\{\alpha_1,\dots,\alpha_n\}<\alpha\vel\beta}{  \ergo\{\alpha_1,\dots,\alpha_n\}<\alpha}
\qquad
\infer{  \ergo\{\beta_1,\dots,\beta_n\}<\alpha\vel\beta}{  \ergo\{\beta_1,\dots,\beta_n\}<\beta}
\]
\[
\infer[\ast]{    \ergo\{\alpha_1,\dots,\alpha_n,\beta_1,\dots,\beta_m\}<\alpha\vel\beta}{  \ergo\{\alpha_1,\dots,\alpha_n\}<\alpha&\ergo\{\beta_1,\dots,\beta_m\}<\beta}
\]
\[\infer{  \ergo\{\alpha_1,\dots,\alpha_n\}<\non(\alpha\et\beta)}{  \ergo\{\alpha_1,\dots,\alpha_n\}<\non\alpha}
\qquad
\infer{  \ergo\{\beta_1,\dots,\beta_n\}<\non(\alpha\et\beta)}{  \ergo\{\beta_1,\dots,\beta_n\}<\non\beta}
\]
\[\infer[\ast]{\ergo\{\alpha_1,\dots,\alpha_n,\beta_1,\dots,\beta_m\}<\non(\alpha\et\beta)}{  \ergo\{\alpha_1,\dots,\alpha_n\}<\non\alpha&\Delta\ergo\{\beta_1,\dots,\beta_m\}<\non\beta}
\]
\[
\infer{    \ergo\{\alpha_1,\dots,\alpha_n,\beta_1,\dots,\beta_m\}<\non(\alpha\vel\beta)}{  \ergo\{\alpha_1,\dots,\alpha_n\}<\non\alpha&\Delta\ergo\{\beta_1,\dots,\beta_m\}<\non\beta}
\]
\[ 
\infer{  \ergo\{\alpha_1,\dots, \alpha_n\}<\non\non\alpha}{  \ergo\{\alpha_1,\dots, \alpha_n\}<\alpha}
\]

\medskip
\hrule
\caption{Introduction rules.}
\label{fig:intro-rules-gr}
\end{figure}

\begin{figure}[ht]
\hrule
\medskip

\[
\infer{  \ergo\{\alpha_1,\dots,\beta,\gamma,\dots, \alpha_n\}<\delta}{  \ergo\{\alpha_1,\dots,\beta\et\gamma,\dots, \alpha_n\}<\delta}
\]
\[ 
\infer{  \ergo\{\alpha_1,\dots,\beta,\dots, \alpha_n\}<\delta}{  \ergo\{\alpha_1,\dots,\beta\vel\gamma,\dots, \alpha_n\}<\delta}
\qquad
\infer{  \ergo\{\alpha_1,\dots,\gamma,\dots, \alpha_n\}<\delta}{  \ergo\{\alpha_1,\dots,\beta\vel\gamma,\dots, \alpha_n\}<\delta}
\]
\[\infer[\ast]{  \ergo\{\alpha_1,\dots,\beta,\gamma, \dots, \alpha_n\}<\delta}{  \ergo\{\alpha_1,\dots,\beta\vel\gamma,\dots, \alpha_n\}<\delta}
\]
\[ 
\infer{  \ergo\{\alpha_1,\dots,\non\beta,\dots, \alpha_n\}<\delta}{  \ergo\{\alpha_1,\dots,\non(\beta\et \gamma),\dots, \alpha_n\}<\delta}
\qquad
\infer{  \ergo\{\alpha_1,\dots,\non\gamma,\dots, \alpha_n\}<\delta}{  \ergo\{\alpha_1,\dots,\non(\beta\et \gamma),\dots, \alpha_n\}<\delta}
\]
\[\infer[\ast]{\ergo\{\alpha_1,\dots,\non\alpha,\non\beta,\dots, \alpha_n\}<\delta}{  \ergo\{\alpha_1,\dots,\non(\beta\et \gamma),\dots, \alpha_n\}<\delta}\]
\[\infer{\ergo\{\alpha_1,\dots,\non\beta,\non\gamma,\dots, \alpha_n\}<\delta}{  \ergo\{\alpha_1,\dots,\non(\beta\vel \gamma),\dots, \alpha_n\}<\delta}
\]
\[ 
\infer{  \ergo\{\alpha_1,\dots,\beta,\dots, \alpha_n\}<\gamma}{  \ergo\{\alpha_1,\dots,\non\non\beta,\dots, \alpha_n\}<\gamma}
\]

\medskip
\hrule
\caption{Elimination rules.}
\label{fig:elim-rules-gr}
\end{figure}

The system $\calcg$ is defined by the rules in Figures \ref{fig:0-rules-gr},  \ref{fig:intro-rules-gr} ,  \ref{fig:elim-rules-gr}  and \ref{fig:structural-rules-gr}. The first four $0$-premiss rules are characterized by the fact that the formula to the left of $<$ contains the immediate subformulas of the formula to the right of $<$. Quite simply, they say that $\{\alpha, \beta\}$ grounds the conjunction $\alpha\et \beta$ and any subset of $\{\alpha, \beta\}$ grounds the disjunction $\alpha\vel \beta$. We mark by $\ast$ the rules enabling us to ground a disjunction by both disjuncts because we will treat separately the calculus containing them and the one not containing them. The rest of the $0$-premiss rules are justified by similar principles but also involve negation. A single negation, nevertheless, cannot be handled just as an occurrence of any other connective. Indeed, rather obviously, there is no way to ground the truth of a formula $\non \alpha$ in the truth of its immediate subformula $\alpha$. Therefore, the rules for ground need to be defined in such a way that they handle an occurrence of negation in combination with an occurrence of some other connective. In order to ground the truth of a negated conjunction $\non(\alpha\et\beta)$, it is enough to use either $\non\alpha$ or $\non\beta$; in order to ground the truth of a negated disjunction $\non(\alpha\vel\beta)$, we must use both elements of the set $\{\non\alpha ,\non\beta\}$. The interaction of negation with itself is rather simple: the ground of $\non\non\alpha$ is simply $\alpha$. The latter, indeed, is the largest subformula of $\non\non\alpha$ that can justify $\non\non\alpha$. The rules in Figure \ref{fig:intro-rules-gr} enable us to introduce an occurrence of a connective in the formula to the right of $<$. The rules in Figure \ref{fig:elim-rules-gr} enable us to eliminate an occurrence of a connective in a formula to the left of $<$.

\begin{figure}[ht]
\hrule
\medskip

{\small 
\[ 
\vcenter{\infer[\set c]{  \ergo\{\beta_1,\dots,\beta_m\}<\alpha}{\ergo\{\alpha_1,\dots,\alpha_n\}<\alpha}}
\]}
{\small \[ 
\infer[\amal]{  \ergo\{\alpha_1,\dots, \alpha_n,\beta_1,\dots, \beta_m\}<\alpha}{ \ergo\{\alpha_1,\dots, \alpha_n\}<\alpha& \ergo\{\beta_1,\dots, \beta_m\}<\alpha}
\]

}

\medskip
\hrule
\caption{Structural rules}
\label{fig:structural-rules-gr}
\end{figure}

The rules in Figure \ref{fig:structural-rules-gr} are basic structural rules. $\set c$ enables us to handle sets in such a way that the order and number of occurrences of their elements do not matter: here the list $\beta_1,\dots,\beta_m$ contains the same elements as $\alpha_1,\dots,\alpha_n$ but without repetitions and, possibly, in a different order. $\amal$---for {\it Amalgamation}---enables us to merge possibly different grounds $\{\alpha_1,\dots, \alpha_n\}$ and $\{\beta_1 , \dots , \beta_m\}$ of the same formula $\alpha$ in case they are indeed grounds of $\alpha$. 



Derivations in $\calcg$ are defined as usual: by applying a rule to derivations of its premisses, one obtains a derivation of the conclusion.

Even though the rules of $\calcg$ do not use contexts to express the dependencies of the derived expressions, we keep the notation $\ergo \alpha$ in order to stress that, technically, we handle $<$ as a connective of the language and not as a relation on the language.\footnote{Nevertheless, $<$ exactly captures the derivability relation induced by calculi such as---if we exclude the $\amal$ rule---the one in \cite{cor14}. Indeed, the conclusions of the $0$-premiss rules of $\calcg$ exactly  capture individual applications of the rules in Correia's calculus, introduction rules in $\calcg$ have the same effect as Correia's rules, and elimination rules in $\calcg$ have the same effect as applying these rules to derive a hypothesis of an existing derivation.}

\section{Direct characterisation of grounds and decidability of the grounding relation} 
 We show that it is possible to characterise the set of all formulae $\{\alpha_1, \dots,\alpha_n \}<\alpha$ which are provable in $\calcg$ by a simple procedure on the syntactic tree of $\alpha$. This implies, in particular, that it is possible determine when a derivable statement $\{\alpha_1, \dots,\alpha_n \}\condr\alpha$ corresponds to a provable grounding claim.

Let us begin by fixing some terminological conventions and by defining some useful notions. In the following, we assume that the root of a tree is on top of it, and hence that a node $a$ is below a node $b$ if, and only if, the path from $a$ to the root of the tree goes through $b$.
 Moreover, for any logical operator $\star$, we call $\star$-nodes all nodes of the syntactic tree of the formula $\alpha$ that correspond to a subformula of $\alpha$ with $\star$ as  main connective.

\begin{definition}[Positive and negative nodes of a syntactic tree]
Given any syntactic tree, we call {\it positive nodes} those that occur strictly below an even number of $\non$-nodes and {\it negative nodes} those that occur strictly below an odd number of $\non$-nodes.
\end{definition}

\begin{definition}[Feeble nodes of a syntactic tree]
Given any syntactic tree $T$, we call {\it feeble nodes} the positive $\vel$-nodes and the negative $\et$-nodes.
\end{definition}



%

%

\begin{definition}[Selection tree]
Given a formula $\alpha$ with syntactic tree $T$, a {\it selection tree} for $\alpha$ is any tree obtained by selecting {\it at most one} child of each feeble node of $T$ ({\it exactly one} child in case the rules for disjunction marked with $\ast$ are not in the calculus), and then by deleting from $T$ all subtrees rooted at any of the selected nodes.
\end{definition}

We remind the reader that a {\it bar} of a tree $T$ is a set $B$ of nodes of $T$ such that exactly one element of $B$ occurs on each path from a leaf of $T$ to its root. We call a bar {\it non trivial} if it does not contain the root of $T$.

\begin{definition}[Grounding bar]
Given any formula $\alpha$, a {\it grounding bar} $G$ for $\alpha$ is obtained as follows: pick a selection tree $S$ for $\alpha$; pick a bar $B$ of $S$ that is non trivial and that does not contain the only child of the root of $S$, if there is one such node; prefix all negative nodes in $B$ by $\non$.
\end{definition}

\begin{theorem}
A grounding claim $\{\alpha_1, \dots , \alpha_n\}<\alpha $ is provable in $\calcg$ if, and only if, the set of formulae $\{\alpha_1, \dots , \alpha_n\}$ is the union of grounding bars for $\alpha$. 
\end{theorem}

Intuitively, constructing the selection tree corresponds to choosing one disjunct for each disjunction and one conjunct for each formula $\non(\beta\et\gamma)$ that might be encountered in a derivation in $\calcg$. In presence of the $\ast$ rules for disjunction, it is possible to choose both disjuncts and both conjuncts. Note that the definition of grounding bar guarantees that a suitable number of negations are prefixed to each subformula of $\alpha$ that occurs in the ground. The union of grounding bars should be considered since the $\amal$ rule is in the calculus. If the $\amal$ rule is not in the calculus, the following, simpler version of the previous theorem holds.
\begin{theorem}
A grounding claim $\{\alpha_1, \dots , \alpha_n\}<\alpha $ is provable in $\calcg$ if, and only if, the set of formulae $\{\alpha_1, \dots , \alpha_n\}$ is a grounding bar for $\alpha$. 
\end{theorem}

Note, moreover, that this characterisation implies that the grounding relation defined by the considered calculi is decidable. Indeed, in order to list all grounds of a formula $\alpha$ it is enough to construct its syntactic tree $T$; construct the set of $S$ all selection trees of $T$, which is of finite cardinality; construct the set $B$ of all grounding bars of elements of $S$, which is finite as well since there exists only a finite number of bars for each finite tree. Since each step produces a finite set of objects, it is always possible to execute all these construction steps in finite time and obtain all grounds of $\alpha$ for any formula $\alpha$.

\subsection{Soundness an completeness of the characterisaiton}

We first prove that, given any
union $\{\alpha_1, \dots , \alpha_n\}$ of grounding bars $G_1,\dots, G_n$ for $\alpha$, we can construct a proof in $\calcg$ of  $\{\alpha_1, \dots , \alpha_n\}<\alpha $.

\begin{theorem}\label{thm:bar-deriv}
Given any
union $\{\alpha_1, \dots , \alpha_n\}$ of grounding bars $G_1,\dots, G_m$ for $\alpha$, the formula $\{\alpha_1, \dots , \alpha_n\}<\alpha $ is provable in $\calcg$.
\end{theorem}
\begin{proof}

Let us consider any set of formulae $\{\alpha_1, \dots , \alpha_n\}= G_1\cup\dots\cup G_m$ such that $G_1,\dots, G_m$ are grounding bars for $\alpha$. If, for any $G\in\{G_1,\dots, G_m\}$ we construct a proof in $\calcg$ of $\{\gamma_1, \dots ,\gamma_k \}<\alpha$ for $\{\gamma_1, \dots ,\gamma_k \}=G$, then we can construct a proof of $\{\alpha_1, \dots , \alpha_n\}<\alpha $ by applying $m$ times the rule $\amal$ to the conclusions of these proofs. Let us then show how to construct the proof of $\{\gamma_1, \dots ,\gamma_k \}<\alpha$ for $\{\gamma_1, \dots ,\gamma_k \}=G$. 

In the following, we will assume that the rules $\ast$ for disjunction are not in the calculus. In case they are, it is enough to use them whenever the selection tree of $\alpha$ on which $G$ is based contains both children of a feeble node.

Let us call $S$ the selection tree for $\alpha$ on which $G$ is based.
 
We start the proof by applying the $0$-premiss rule with conclusion $\{ \dots \}<\alpha$ such that the ground $\{ \dots \}$ contains the children of the root of $S$, possibly prefixed by $\non$ if the rule requires it. Now, we pick one element $\gamma$ of the ground occurring at the conclusion of the obtained derivation and apply one elimination rule to it such that the ground at the conclusion of the rule application contains---instead of $\gamma$---the children of the node for $\gamma$ in $S$, possibly prefixed by $\non$ if the rule requires it. We never apply rules to elements of $G$ that occur in the ground at the conclusion of the derivation and we stop when this ground only contains elements of $G$. 

To see that this always produces the desired proof, it is enough to note that the definition of selection tree simply enables us to remove subtrees of a syntactic tree which are rooted at nodes such that, when we can apply an elimination rule to the corresponding formula, we need to choose one of two possible subformulae to be kept in the ground. The selection tree makes that choice. This choice cannot be simply made for $\et$-nodes and for nodes of the form $\non(\dots \vel \dots )$ but must be made for positive $\et$-nodes and negative $\vel$-nodes since, when the formula corresponding to the node will appear in a grounding derivation, it will appear prefixed with $\non$ in case the corresponding node is negative. Note, moreover, that this fact concerning negations is also considered in the definition of grounding bar.
\end{proof}


We now prove that, given any derivation in $\calcg$  of $\{\alpha_1, \dots , \alpha_n\}<\alpha $, we can show that $\{\alpha_1, \dots , \alpha_n \}$ is the union of some grounding bars for $\alpha$.
\begin{theorem}\label{thm:deriv-bar}
If $\{\alpha_1, \dots , \alpha_n\}<\alpha $ is provable in $\calcg$, then there exist grounding bars $G_1\cup\dots\cup G_m$ for $\alpha$ such that $\{\alpha_1, \dots , \alpha_n \}=G_1\cup\dots\cup G_m$.
\end{theorem}
\begin{proof}
The proof in by induction on the rules applied in the derivation of $\{\alpha_1, \dots , \alpha_n\}<\alpha $. 

In the following, we will assume that the rules $\ast$ for disjunction are not in the calculus. In case they are, it is enough to use the definition of selection tree that admits the case in which no child of a feeble node is selected.

If the only rule in the derivation is a $0$-premiss rule, it is easy to see that the ground is a grounding bar for $\alpha$. We consider three exemplar cases.

A rule application of the form
\[\infer{  \ergo\{\beta,\gamma\}<\beta\et\gamma}{}
\]corresponds to the grounding bar $\{\beta,\gamma\}$ obtained from the bar $\{\beta,\gamma\}$ of any selection tree for $\beta\et\gamma$. The node $\beta\et\gamma$ in the syntactic tree of $\beta\et\gamma$ is, indeed, not feeble, and hence we do not remove any of its children.

A rule application of the form
\[\infer{  \ergo\{\non\beta\}<\non(\beta\et\gamma)}{}
\]corresponds to a grounding bar obtained from the bar $\{\beta\}$ of the selection tree for $\non(\beta\et\gamma)$ in which we removed the subtree rooted at the node $\gamma$ of the syntactic tree of $\beta$. The node $\beta\et\gamma$ in the syntactic tree of $\non(\beta\et\gamma)$ is, indeed, feeble because it is a negative $\et$-node. 
Since, moreover, also the node $\beta$ is negative, any grounding bar for $\non(\beta\et\gamma)$ based on a bar containing $\beta$ contains $\non\beta$.

A rule application of the form
\[ 
\infer{  \ergo\{\non\beta,\non\gamma\}<\non(\beta\vel\gamma)}{}
\]corresponds to the grounding bar $\{\non\beta,\non\gamma\}$ obtained from the bar $\{\beta,\gamma\}$ of any selection tree for $\non(\beta\vel\gamma)$. The node $\beta\vel\gamma$ in the syntactic tree of $\non(\beta\vel\gamma)$ is, indeed, not feeble because it is a negative $\vel$-node. Hence we do not remove any of its children. Since the nodes $\beta$ and $\gamma$ are negative, any grounding bar for $\non(\beta\vel\gamma)$ contain $\non\beta,\non\gamma$.


Let us now suppose that all derivations containing less than $k$ rule applications comply with the statement of the theorem, we prove that also any derivation containing $k>1$ rule applications complies with the statement.

Let us reason on the last rule applied in the derivation.

Suppose the last rule applied in the derivation is an elimination rule:
\[\infer{  \ergo\{\beta_1,\dots,\gamma_1, (\gamma_2,)\dots, \beta_z\}<\delta}{  \ergo\{\beta_1,\dots,\gamma,\dots, \beta_z\}<\delta}
\]where $\gamma_2$ might not occur. By inductive hypothesis, $\{\beta_1,\dots,\gamma,\dots, \beta_z\}$ corresponds to a grounding bar for $\delta$. 
The 
ground $\{\beta_1,\dots,\gamma_1, (\gamma_2,)\dots, \beta_z\}$ corresponds to the grounding bar $G$ based on the selection tree $S$ for $\delta $ where 
\begin{itemize}
\item if the node of the syntactic tree of $\delta$ corresponding to $\gamma$ is feeble, $S$ is obtained by removing the subtree rooted at the child of this node that corresponds to the subformula of $\gamma$ that has been eliminated in order to obtain $\{\beta_1,\dots,\gamma_1, (\gamma_2,)\dots, \beta_z\}$, and 
\item $G$ is obtained from $S$ by replacing the node corresponding to $\gamma$ with all its children.
\end{itemize}
If $\gamma$ is prefixed by  $\non$ while the corresponding node is not, then also $\gamma_1, (\gamma_2) $ will be prefixed by $\non$.

Suppose the last rule applied in the derivation is a $1$-premiss introduction rule:
\[\infer{\ergo\{\beta_1,\dots,\beta_z\}<\gamma}{\ergo\{\beta_1,\dots,\beta_z\}<\gamma_1}
\]By inductive hypothesis, $\{\beta_1,\dots,\beta_z\}$ corresponds to a grounding bar $G$ for $\gamma_1$. We define a selection tree $S$ for $\gamma$ as follows:
we use $\gamma$ as root, append one or two nodes corresponding to $\gamma_1$ (two if $\gamma$ is negated) below the root, and we append the selection tree on which $G$ is based below the nodes corresponding to $\gamma_1$. The nodes corresponding to $\gamma_1$ are obtained, in case the introduction application under consideration is one for negation, by removing the outermost $\non$ from $\gamma$  and then appending to this node the formula obtained 
by removing the outermost $\non$ from  $\gamma_1$; in case the introduction application under consideration is not for negation, by simply appending $\gamma_1$ to $\gamma$. It is easy to see that $G$ is a grounding bar for $S$. Indeed, the positive nodes in the syntactic tree of $\gamma_1$ are positive also in the syntactic tree of $\gamma$ and the negative nodes in the former are negative also in the latter.

Suppose the last rule applied in the derivation is a $2$-premiss introduction rule:
\[\infer{\ergo\{\beta_1,\dots,\beta_z,\gamma_1,\dots,\gamma_x\}<\delta}{\ergo\{\beta_1,\dots,\beta_z\}<\delta_1 & \ergo\{\gamma_1,\dots,\gamma_x\}<\delta_2}
\]By inductive hypothesis, the grounds occurring in the premisses correspond to grounding bars $G_1$ and $G_2$ for $ \delta_1$ and $\delta_2$ respectively. We define a selection tree $S$ for $\delta$ as follows:
we use $\delta$ as root, append some nodes  corresponding to $ \delta_1,\delta_2$ (three if $\delta$ is negated) below the root, and we append the selection trees on which $G_1$ and $G_2$ are based below the nodes corresponding to $\delta_1$ and $\delta_2$ respectively. The nodes corresponding to $\delta_1,\delta_2$ are obtained, in case the introduction application under consideration is one for negation, by removing the outermost $\non$ from $\delta$  and then appending to this node the formulae obtained by removing the outermost $\non$ from $\delta_1$ and $\delta_2$; in case the introduction application under consideration is not for negation, by simply appending $\delta_1$ and $\delta_2$ as children of $\delta$. It is easy to see that $G_1\cup G_2$ is a grounding bar for $S$. Indeed, the positive nodes in the syntactic trees of $\delta_1$ and $\delta_2$ are positive also in the syntactic tree of $\delta$ and the negative nodes in the former trees are negative also in the latter tree.

If the last rule applied in the derivation is $\set c$, there is nothing to prove since the set of formulae corresponding to the ground does not change.
%

If the last rule applied in the derivation is $\amal$, it is enough to take the union of the two minimal logical decomposition corresponding to the premisses. These exist by inductive hypothesis.

 \[ 
\infer[\amal]{  \ergo\{\alpha_1,\dots, \alpha_n,\beta_1,\dots, \beta_m\}<\alpha}{ \ergo\{\alpha_1,\dots, \alpha_n\}<\alpha& \ergo\{\beta_1,\dots, \beta_m\}<\alpha}
\]

\end{proof}

\bibliographystyle{apalike}
\bibliography{bib-raga.bib}
\end{document}